\documentclass[reqno,11pt]{amsart}
\usepackage[foot]{amsaddr}
\usepackage{bbold}
\usepackage{charter}
\usepackage[margin=1in]{geometry}
\usepackage[colorlinks=true,linktoc=all,linkcolor=blue,citecolor=blue]{hyperref}

\theoremstyle{plain}
\newtheorem{theorem}{Theorem}[section]
\newtheorem*{theorem*}{Theorem}
\newtheorem{proposition}{Proposition}[theorem]
\newtheorem{lemma}[theorem]{Lemma}
\newtheorem{corollary}[theorem]{Corollary}
\theoremstyle{definition}

\newtheorem*{definition*}{Definition}
\newtheorem{remark}[theorem]{Remark}

\newtheorem{conjecture}[theorem]{Conjecture}
\numberwithin{equation}{section}
\everymath{\displaystyle}
\allowdisplaybreaks

\newcommand{\B}{\mathbb{B}}
\newcommand{\C}{\mathbb{C}}
\newcommand{\D}{\mathbb{D}}
\newcommand{\N}{\mathbb{N}}

\newcommand{\Bloch}{\mathcal{B}}
\newcommand{\lilBloch}{\mathcal{B}_0}
\newcommand{\norm}[1]{\left|\left|#1\right|\right|}
\newcommand{\supnorm}[1]{\norm{#1}_\infty}
\newcommand{\blochnorm}[1]{\norm{#1}_\Bloch}
\newcommand{\ip}[1]{\left\langle #1 \right\rangle}
\newcommand{\conj}[1]{\overline{#1}}
\newcommand{\Aut}{\mathrm{Aut}}
\newcommand{\Log}{\mathrm{Log}}
\renewcommand{\mod}[1]{\left|#1\right|}

\title[Weighted Composition Operators]
{Weighted Composition Operators on the Bloch Space\\ of a Bounded Homogeneous Domain}

\author{Robert F. Allen\textsuperscript{1} and Flavia Colonna\textsuperscript{2}}
\address{\textsuperscript{1}Department of Mathematics, University of Wisconsin--La Crosse}
\address{\textsuperscript{2}Department of Mathematical Sciences, George Mason University}
\email{allen.rob3@uwlax.edu, fcolonna@uwlax.edu}

\dedicatory{In honor of Israel Gohberg on the occasion of his 80th birthday.}

\begin{document}

\begin{abstract} In this paper, we present the current results in the study of weighted composition operators on the Bloch space of bounded homogeneous domains in $\C^n$ with particular emphasis on the issues of boundedness and compactness.  We also discuss the bounded and the compact weighted composition operators from the Bloch space to the Hardy space $H^\infty$.
\end{abstract}

\maketitle

\section{Introduction}\label{section:introduction}
Let $X$ be a Banach space of holomorphic functions on a domain $\Omega \subset \C^n$.  For $\psi$ a holomorphic function on $\Omega$ and $\varphi$ a holomorphic self-map of $\Omega$, the linear operator defined by $$W_{\psi,\varphi}(f) = \psi(f\circ\varphi),\quad f \in X,$$ is called the \emph{weighted composition operator} with symbols $\psi$ and $\varphi$.  Observe that $W_{\psi,\varphi}(f) = M_\psi C_\varphi(f)$ where $M_\psi(f) = \psi f$ is the \emph{multiplication operator} with symbol $\psi$ and $C_\varphi(f) = f\circ\varphi$ is the \emph{composition operator} with symbol $\varphi$.  If $\psi$ is identically 1, then $W_{\psi,\varphi} = C_\varphi$, and if $\varphi$ is the identity, then $W_{\psi,\varphi} = M_\psi$.

The study of weighted composition operators is fundamental in the study of Banach and Hilbert spaces of holomorphic
functions.  The study of the geometry of a space $X$ is centered on the identification of the isometries on $X$.  The connection between weighted composition operators and isometries can be traced back to Banach himself.  In \cite{Banach:32}, Banach proved that the surjective isometries on $C(Q)$, the space of continuous real-valued functions on a compact metric space $Q$, are of the form $Tf = \psi(f\circ\varphi),$ where $\mod{\psi} \equiv 1$ and $\varphi$ is a homeomorphism of $Q$ onto itself.

Although the characterization of isometries is an open problem for most Banach spaces of holomorphic functions, there are many spaces for which the isometries are known.  In \cite{Forelli:64}, Forelli proved that the isometries on the Hardy space $H^p$ of the open unit disk $\D$ (for $p \neq 2$) are certain weighted composition operators.  On the Bergman space $A^p$ of $\D$, Kolaski showed that the surjective isometries are weighted composition operators \cite{Kolaski:82}.  El-Gebeily and Wolfe showed that the isometries on the disk algebra are also weighted composition operators \cite{ElGebeilyWolfe:85}. Thus the weighted composition operator plays a central role in the study of the isometries on several spaces of holomorphic functions.

The first study of the isometries on the Bloch space was made by Cima and Wogen in \cite{CimaWogen:80}.  They analyzed the isometries on the subspace of the Bloch space of the open unit disk whose elements fix the origin. On this space, they showed that the surjective isometries are normalized compressions of weighted composition operators induced by disk automorphisms. In \cite{KrantzMa:1989}, Krantz and Ma extended their results to the Bloch space of the unit ball in $\C^n$. However, in any dimension, a description of all isometries on the entire set of Bloch functions is still unknown.

The study of weighted composition operators is not limited to the study of isometries.  The properties of the weighted composition operators are not solely determined by the component operators, namely multiplication and composition operators.  Indeed, there exist bounded weighted composition operators on the Bloch space for which the associated multiplication operator is not bounded. Likewise, there are compact weighted composition operators for which neither component operator is compact. Examples of such operators were provided by Ohno and Zhao in \cite{OhnoZhao:01} in the one-dimensional case.  In Sections 5 and 6, we give analogous examples for the unit ball and the unit polydisk in $\C^n$. Thus, the study of weighted composition operators is truly an evolutionary step in the field of composition operators.

\subsection{Purpose of the paper}
From the previous discussion, it is clear that the study of weighted composition operators is a worthwhile endeavor.  A primary purpose of this paper is to bring the current results on the weighted composition operators on the Bloch space to one location.  There are still many open questions, and thus opportunities for active research.  Thus, our hope is that this exposition will inspire more work in this area.  To this end, we add to this paper some new results, accompanied by some conjectures and areas for future research.

\subsection{Organization of the paper}
In Section 2, we review the notion of the Bloch space on the unit disk $\D$ and on bounded homogeneous domains.  In Section 3, we outline the results known on weighted composition operators on the Bloch space and little Bloch space of $\D$.  These include the characterization of the bounded and the compact operators due to Ohno and Zhao and operator norm estimates.

In Section 4, we present the known results on the weighted composition operators on the Bloch space in higher dimensions.  For a bounded homogeneous domain $D$ we define quantities which we believe are proper candidates to characterize the bounded and the compact weighted composition operators on the Bloch space of $D$ and on a subspace we refer to as the {\it $*$-little Bloch space}, which is a higher-dimensional analogue of the little Bloch space. We give sufficient conditions for boundedness and compactness and give operator norm estimates.

In Sections 5 and 6, we prove the conjectures presented in Sections~4 and 5 for the Bloch space on the unit ball and unit polydisk which yield results equivalent to Corollaries 1.4 and 1.6 of \cite{ZhouChen:05} and Theorems 1 and 2 of \cite{ZhouChen-I:05}.

In Section~7, we characterize the bounded weighted composition operators from the Bloch space and the $*$-little Bloch space into the space of bounded holomorphic functions on a bounded homogeneous domain and determine the norm of such operators. As a special case, we obtain Theorem~1 of \cite{LiStevic:08}. We also prove an extension of Theorem~6.1 of
\cite{HosokawaIzuchiOhno:05} to the unit polydisk.

Finally, in Section 8 we discuss further developments and open problems for the weighted composition operators on the Bloch space of a bounded homogeneous domain.

\section{The Bloch Space}\label{section:bloch space}
The Bloch space has been defined on many types of domains in $\C^n$.  The first such domain we will consider is the open unit disk $\D$. A complex-valued function $f$ analytic on $\D$ is said to be \emph{Bloch} if $$\beta_f := \sup_{z \in \D}\;(1-\mod{z}^2)\mod{f'(z)} < \infty.$$  The mapping $f \mapsto \beta_f$ is a semi-norm on the space $\Bloch(\D)$ of Bloch functions on $\D$ and $\Bloch(\D)$ is a Banach space under the Bloch norm $$\blochnorm{f} = \mod{f(0)} + \beta_f.$$

By the Schwarz-Pick lemma, the space $H^\infty(\D)$ of bounded analytic functions on $\D$ is a subset of $\Bloch(\D)$ and the containment is proper, since $z \mapsto \Log (1-z)$ is a Bloch function, where $\Log$ denotes the principal branch of the logarithm.  The \emph{little Bloch space} $\lilBloch(\D)$ on $\D$ is defined as the set of Bloch functions $f$ such that $$\lim_{\mod{z} \to 1}\;(1-\mod{z}^2)\mod{f'(z)} = 0.$$  The little Bloch space is a separable subspace of $\Bloch(\D)$ since the polynomials form a dense subset.  Useful references on Bloch functions and the Bloch space on $\D$ include \cite{Pommerenke:70}, \cite{AndersonCluniePomerenke:74} and \cite{Cima:79}.

As an immediate consequence of the Schwarz-Pick lemma, if $f \in \Bloch(\D)$ and $\varphi$ is an analytic self-map of $\D$, then $f\circ\varphi \in \Bloch(\D)$.  Furthermore, $\beta_f = \beta_{f\circ\varphi}$ for any conformal automorphism $\varphi$ of $\D$, that is, the Bloch space is \emph{M\"obius invariant}.  In fact, it is the largest M\"obius invariant Banach space \cite{Tjani:96}.

The notion of Bloch function in higher dimensions was introduced by Hahn in \cite{Hahn:75}.  In \cite{Timoney:80-I} and \cite{Timoney:80-II} Timoney studied extensively the space of Bloch functions on a bounded homogeneous domain and its subspace known as the little Bloch space on a bounded symmetric domain.

Every bounded domain $D \subset \C^n$ is endowed with a canonical metric called the \emph{Bergman metric}, which is invariant under the action of the group of biholomorphic transformations, which we call \emph{automorphisms} and denote by $\Aut(D)$ \cite{Helgason:62}.  We will focus on a particular class of domains in $\C^n$, the homogeneous domains.  A domain $D$ in $\C^n$ is called \emph{homogeneous} if $\Aut(D)$ acts transitively on $D$, that is, for all $z_1,z_2 \in D$, there exists $\phi \in \Aut(D)$ such that $\phi(z_1) = z_2$.

A domain $D \subset \C^n$ is \emph{symmetric at a point $z_0 \in D$} if there exists $\phi \in \Aut(D)$ such that $\phi\circ\phi$ is the identity and $z_0$ is an isolated fixed point of $\phi$.  A domain is \emph{symmetric} if it is symmetric at each of its points.  A symmetric domain is homogeneous and a homogeneous domain that is symmetric at a single point is symmetric.  Therefore the unit ball $\B_n$ and the unit polydisk $\D^n$ are symmetric, since they are homogeneous and symmetric at the origin via $z \mapsto -z$.

Let $D$ be a bounded homogeneous domain in $\C^n$.  A holomorphic function $f:D \to \C$ is said to be a Bloch function if $\beta_f = \displaystyle\sup_{z\in D}\;Q_f(z)$ is finite, where \begin{equation}\nonumber Q_f(z) = \sup_{u\in \C^n\setminus\{0\}}\;\frac{|(\nabla f)(z)u|}{H_z(u,\overline{u})^{1/2}},\end{equation}
$(\nabla f)(z)u = \ip{\nabla f(z),\conj{u}} = \displaystyle\sum_{k=1}^n \frac{\partial f}{\partial z_k}(z)u_k$, and $H_z$ is the Bergman metric on $D$ at $z$.  By fixing a base point $z_0 \in D$, the Bloch space $\Bloch(D)$ is a Banach space under the norm $\blochnorm{f} = \mod{f(z_0)} + \beta_f$ \cite{Timoney:80-I}.  For convenience, we assume the domain $D$ to contain the origin and take $z_0 = 0$.

In \cite{Timoney:80-I}, Timoney proved that the space $H^\infty(D)$ of bounded holomorphic functions on a bounded homogeneous domain $D$ is a subspace of $\Bloch(D)$ and for each $f \in H^\infty(D)$, $\blochnorm{f} \leq c_D\supnorm{f}$ where $c_D$ is a constant depending only on the domain $D$.  The precise value of the best bound $c_D$ has been calculated in \cite{CohenColonna:94} and \cite{Zhang:97} when $D$ is a bounded symmetric domain.

In Theorem 3.1 of \cite{AllenColonna:07}, we showed that the Bloch functions on $D$ are precisely the Lipschitz maps between the metric spaces $D$ and $\C$ under the Bergman metric and Euclidean metric, respectively.  Furthermore \begin{equation}\label{beta_f alternative}\beta_f = \sup_{z \neq w}\;\frac{\mod{f(z)-f(w)}}{\rho(z,w)},\end{equation} where $\rho$ is the Bergman distance.  In particular, for all $z, w \in D$, \begin{equation}\label{|f(z)-f(w)| < rho}\mod{f(z)-f(w)} \leq \blochnorm{f}\rho(z,w).\end{equation}

In \cite{Timoney:80-II} the little Bloch space on the unit ball was defined as
$$\lilBloch(\B_n) = \left\{f \in \Bloch(\B_n) : \lim_{\norm{z} \to 1} Q_f(z) = 0\right\},$$  which is precisely the closure of the polynomials in $\Bloch(\B_n)$.  If $D$ is a bounded symmetric domain in $\C^n$ other than $\B_n$, the set of functions $f$ for which $Q_f(z) \to 0$ as $z$ approaches the boundary $\partial D$ of $D$ consists only of the constant functions, so $\lilBloch(D)$ is defined as the closure of the polynomials in $\Bloch(D)$.  The \emph{$*$-little Bloch space} is defined as $$\Bloch_{0^*}(D) = \left\{f \in \Bloch(D) : \lim_{z \to \partial^*\hskip-2pt D}\;Q_f(z) = 0\right\},$$ where $\partial^*\hskip-2pt D$ denotes the distinguished boundary of $D$.  The unit ball is the only bounded symmetric domain $D$ for which $\partial D = \partial^*\hskip-2pt D$, so that $\lilBloch(\B_n) = \Bloch_{0^*}(\B_n)$. If $D \neq \B_n$, $\lilBloch(D)$ is a proper subspace of $\Bloch_{0^*}(D)$ and $\Bloch_{0^*}(D)$ is a non-separable subspace of $\Bloch(D)$.

\section{Weighted Composition Operators on the Bloch Space of $\D$}\label{section:wco disk}
The first results on weighted composition operators on the Bloch space of the unit disk were obtained by Ohno and Zhao in 2001 \cite{OhnoZhao:01}.  For $\psi$ an analytic function on $\D$,  $\varphi$ an analytic self-map of $\D$, and $z \in \D$, define $s_{\psi,\varphi} = \displaystyle\sup_{z\in \D}\; s_{\psi,\varphi}(z)$ and $\tau_{\psi,\varphi} = \displaystyle\sup_{z \in \D}\; \tau_{\psi,\varphi}(z)$ where
$$\begin{aligned}
s_{\psi,\varphi}(z) &= (1-\mod{z}^2)\mod{\psi'(z)}\log\frac{2}{1-\mod{\varphi(z)}^2},\\
\tau_{\psi,\varphi}(z) &= \frac{1-\mod{z}^2}{1-\mod{\varphi(z)}^2}\mod{\varphi'(z)}\mod{\psi(z)}.
\end{aligned}$$

\begin{theorem}\label{thm:wco bounded/compact} Let $\psi$ be an analytic function on $\D$ and $\varphi$ an analytic self-map of $\D$. Then
\begin{enumerate}
\item[(a)] {\normalfont{\textbf{(}\cite{OhnoZhao:01}, \textbf{Theorems 1 and 2).}}} $W_{\psi,\varphi}$ is bounded on $\Bloch(\D)$ if and only if $s_{\psi,\varphi}$ and $\tau_{\psi,\varphi}$ are finite.  Furthermore, the bounded operator $W_{\psi,\varphi}$ is compact on $\Bloch(\D)$ if and only if $$\lim_{\mod{\varphi(z)}\to 1}\;s_{\psi,\varphi}(z) = \lim_{\mod{\varphi(z)}\to 1}\;\tau_{\psi,\varphi}(z) = 0.$$

\item[(b)] {\normalfont{\textbf{(}\cite{OhnoZhao:01}, \textbf{Theorems 3 and 4).}}} $W_{\psi,\varphi}$ is bounded on the little Bloch space $\lilBloch(\D)$ if and only if $\psi \in \lilBloch(\D)$, $s_{\psi,\varphi}$ and $\tau_{\psi,\varphi}$ are finite, and $$\lim_{\mod{z}\to 1}\;\mod{\psi(z)}\mod{\varphi'(z)}(1-\mod{z}^2) = 0.$$  Furthermore, the bounded operator $W_{\psi,\varphi}$ is compact on $\lilBloch(\D)$ if and only if $$\lim_{\mod{z}\to 1}\;s_{\psi,\varphi}(z) = \lim_{\mod{z}\to 1}\;\tau_{\psi,\varphi}(z) = 0.$$
\end{enumerate}
\end{theorem}

In \cite{AllenColonna:08}, we established estimates on the norm of the weighted composition operator $W_{\psi,\varphi}$ on $\Bloch(\D)$ in terms of $\tau_{\psi,\varphi}$ and the quantity
$$\sigma_{\psi,\varphi} = \sup_{z \in \D}\;\frac{1}{2}(1-\mod{z}^2)\mod{\psi'(z)}\log\frac{1+\mod{\varphi(z)}}{1-\mod{\varphi(z)}},$$ which is closely related to $s_{\psi,\varphi}$ but is more amenable to a higher dimensional interpretation since the factor $\frac12\log\frac{1+\mod{\varphi(z)}}{1-\mod{\varphi(z)}}$ is precisely the Bergman distance between $0$ and $\varphi(z)$.

\begin{theorem} Let $\psi$ be analytic on $\D$ and $\varphi$ an analytic self-map of $\D$.  Then
\begin{enumerate}
\item[(a)] $W_{\psi,\varphi}$ is bounded on $\Bloch(\D)$ if and only if $\psi \in \Bloch(\D)$, and $\sigma_{\psi,\varphi}$ and $\tau_{\psi,\varphi}$ are finite.  Furthermore,
    \begin{eqnarray}
    \norm{W_{\psi,\varphi}} &\geq& \max\left\{\blochnorm{\psi}, \displaystyle\frac{1}{2}\mod{\psi(0)}\log\frac{1+\mod{\varphi(0)}}{1-\mod{\varphi(0)}}\right\}\label{inequality:norm of WCO lower}\\
    \norm{W_{\psi,\varphi}} &\leq& \max\left\{\blochnorm{\psi}, \displaystyle\frac{1}{2}\mod{\psi(0)}\log\frac{1+\mod{\varphi(0)}}{1-\mod{\varphi(0)}} + \tau_{\psi,\varphi} + \sigma_{\psi,\varphi}\right\}.\label{inequality:norm of WCO upper}
    \end{eqnarray}  Furthermore, $W_{\psi,\varphi}$ is compact if and only if  \begin{equation}\nonumber\lim_{\mod{\varphi(z)} \to 1}\sigma_{\psi,\varphi}(z) = \lim_{\mod{\varphi(z)}\to 1}\tau_{\psi,\varphi}(z) = 0.\end{equation}

\item[(b)] $W_{\psi,\varphi}$ is bounded on $\lilBloch(\D)$ if and only if $\psi \in \lilBloch(\D)$, $\sigma_{\psi,\varphi}$ and $\tau_{\psi,\varphi}$ are finite, and $$\lim_{\mod{z}\to 1}(1-\mod{z}^2)\mod{\psi(z)}\mod{\varphi'(z)} = 0.$$  Inequalities {\normalfont{(\ref{inequality:norm of WCO lower})}} and {\normalfont{(\ref{inequality:norm of WCO upper})}} hold.  Furthermore, $W_{\psi,\varphi}$ is compact if and only if $$\lim_{\mod{z} \to 1}\sigma_{\psi,\varphi}(z) = \lim_{\mod{z}\to 1}\tau_{\psi,\varphi}(z) = 0.$$
\end{enumerate}
\end{theorem}

\begin{proof} Assume $W_{\psi,\varphi}$ is bounded.  Using as a test function the constant 1, we obtain $\psi \in \Bloch(\D)$.  Since $\sigma_{\psi,\varphi} \leq s_{\psi,\varphi}$, by Theorem \ref{thm:wco bounded/compact}, it follows that $\sigma_{\psi,\varphi}$  and $\tau_{\psi,\varphi}$ are finite. The estimates (\ref{inequality:norm of WCO lower}) and (\ref{inequality:norm of WCO upper}) follow from Theorems 2.1 and 2.2 of \cite{AllenColonna:08}. Conversely, assume $\psi \in \Bloch(\D)$, and $\sigma_{\psi,\varphi}$ and $\tau_{\psi,\varphi}$ are finite.  By the calculation carried out in \cite{AllenColonna:08}, $W_{\psi,\varphi}$ maps $\Bloch(\D)$ into itself and estimates (\ref{inequality:norm of WCO lower}) and (\ref{inequality:norm of WCO upper}) hold.  Thus $W_{\psi,\varphi}$ is bounded.  Observing that for each $z \in \D$, $\sigma_{\psi,\varphi(z)} \leq s_{\psi,\varphi}(z)$ and for $\mod{\varphi(z)} \geq \frac{1}{2}$, $s_{\psi,\varphi}(z) \leq 2\sigma_{\psi,\varphi}(z)$, the characterization of the compactness follows at once from Theorem \ref{thm:wco bounded/compact}.  The proof of part (b) is analogous.\end{proof}

\noindent The above estimates agree with the norm estimates for the composition operators on $\Bloch(\D)$ in \cite{Xiong:04} when $\psi$ is taken to be the constant function 1.

\section{Weighted Composition Operators on the Bloch Space of a Bounded Homogeneous Domain}\label{section:wco bhd}
Let $D$ be a bounded homogeneous domain in $\C^n$.  For $z \in D$, define $$\begin{aligned}\omega(z) &= \sup\left\{\mod{f(z)} : f \in \Bloch(D), f(0) = 0 \text{ and } \blochnorm{f} \leq 1\right\},\\
\omega_0(z) &= \sup\left\{\mod{f(z)} : f \in \Bloch_{0^*}(D), f(0) = 0, \text{ and } \blochnorm{f} \leq 1\right\}.\end{aligned}$$

\begin{lemma} Let $D$ be a bounded homogeneous domain in $\C^n$.  For each $z \in D$, $\omega(z)$ and $\omega_0(z)$ are finite.  In fact, $\omega_0(z) \leq \omega(z) \leq \rho(z,0)$.\end{lemma}

\begin{proof} Let $z \in D$.  The inequality $\omega_0(z) \leq \omega(z)$ is obvious.  For $f \in \Bloch(D)$, by (\ref{beta_f alternative}), $\mod{f(z)-f(0)} \leq \rho(z,0)\beta_f.$  By taking the supremum over all $f \in \Bloch(D)$ such that $f(0) = 0$ and $\blochnorm{f} \leq 1$, we have $\omega(z) \leq \rho(z,0)$.\end{proof}

\begin{remark}\label{remark_section_4} By Theorems 3.9 and 3.14 in \cite{Zhu:04}, it follows immediately that for all $z \in \B_n$ $\omega_0(z) = \omega(z) = \rho(z,0)$ where $$\rho(z,0) = \frac{1}{2}\log\frac{1+\norm{z}}{1-\norm{z}}.$$  It is unknown whether there are other domains for which either equality holds.  The following lemma shows the relationship between point evaluation of Bloch functions (respectively, little Bloch functions) and $\omega$ (respectively, $\omega_0$).\end{remark}

\begin{lemma}\label{inequality:f(z) bound bloch} Let $D$ be a bounded homogeneous domain in $\C^n$ and let $f \in \Bloch(D)$ $(\text{respectively}, f \in \Bloch_{0^*}(D))$.  Then for all $z \in \D$, we have $$\mod{f(z)} \leq \mod{f(0)} + \omega(z)\beta_f,$$ $(\text{respectively}, \mod{f(z)} \leq \mod{f(0)} + \omega_0(z)\beta_f)$.\end{lemma}

\begin{proof} Let $f \in \Bloch(D)$.  The result is immediate if $f$ is constant.  For $f$ non-constant and $z\in D$, the function defined by $$g(z) = \frac{1}{\beta_f}(f(z) - f(0))$$ is Bloch and satisfies the conditions $g(0) = 0$ and $Q_g(z) = \frac{1}{\beta_f}Q_f(z)$ for all $z \in D$.  Thus, $\blochnorm{g} = 1$, so $\mod{g(z)} \leq \omega(z)$ for all $z \in D$.  Consequently,
$$\mod{f(z)} \leq \mod{f(0)} + \mod{f(z)-f(0)} = \mod{f(0)} + \mod{g(z)}\beta_f \leq \mod{f(0)} + \omega(z)\beta_f.$$  The proof for the case $f \in \Bloch_{0^*}(D)$ is analogous.
\end{proof}

For $z \in D$, denote by $J\varphi(z)$ the Jacobian matrix of $\varphi$ at $z$ (i.e. the matrix whose $(j,k)$-entry is $\frac{\partial\varphi_j}{\partial z_k}(z)$). Define the \emph{Bergman constant} of $\varphi$ by $B_\varphi = \sup_{z \in D}\; B_\varphi(z)$, where for $z\in D$ $$B_\varphi(z) = \sup_{u \in \C^n\setminus\{0\}}\;\frac{H_{\varphi(z)}(J\varphi(z)u,\conj{J\varphi(z)u})^{1/2}}{H_z(u,\conj{u})^{1/2}}.$$  In \cite{AllenColonna:07}, the Bergman constant was used for the study of composition operators on the Bloch space.  Specifically, for $f \in \Bloch(D)$, \begin{equation}\label{inequality:B_varphi}Q_{f\circ\varphi}(z) \leq B_\varphi(z)Q_f(\varphi(z))\end{equation} for all $z \in D$. Letting
$$\begin{aligned}
T_{0,\varphi}(z) &= \sup\{Q_{f\circ\varphi}(z):\, f \in \Bloch_{0^*}(D), \beta_f \leq 1\},\\
T_\varphi(z) &= \sup\{Q_{f\circ\varphi}(z):\, f\in\Bloch(D), \beta_f\leq 1\},
\end{aligned}$$ from (\ref{inequality:B_varphi}), it follows that
\begin{equation}\label{T-B inequality}T_{0,\varphi}(z) \leq T_\varphi(z)\leq B_\varphi(z)\end{equation} for each $z\in D$. Moreover, for each $f\in\Bloch(D)$ (respectively, $\Bloch_{0^*}(D)$) and $z\in D$,
\begin{eqnarray} Q_{f\circ\varphi}(z) &\leq& T_\varphi(z)\beta_f\label{testimate}\\
(\text{respectively, } Q_{f\circ\varphi}(z) &\leq& T_{0,\varphi}(z)\beta_f).\notag\end{eqnarray}

For $z \in \D$, by Remark \ref{remark_section_4}, we have
\begin{eqnarray} \omega(\varphi(z)) &=& \frac{1}{2}\log\frac{1+\mod{\varphi(z)}}{1-\mod{\varphi(z)}}\ \hbox{ and } \nonumber\\T_{0,\varphi}(z) &=& T_\varphi(z)= \frac{(1-\mod{z}^2)\mod{\varphi'(z)}}{1-\mod{\varphi(z)}^2}=B_\varphi(z),\nonumber\end{eqnarray} since the right-hand side of the above formula equals $(1-|z|^2)|(f\circ \varphi)'(z)|$ for $$f(w)=\frac{\varphi(z)-w}{1-\overline{\varphi(z)}w},\  w\in\D,$$ which is in the little Bloch space.

For a bounded homogeneous domain $D$ in $\C^n$, $\psi$ holomorphic on $D$, and $\varphi$ holomorphic self-map of $D$, we define $$\begin{aligned}\sigma_{\psi,\varphi} &= \sup_{z \in D}\;\omega(\varphi(z))Q_\psi(z)\\ \tau_{\psi,\varphi} &= \sup_{z \in D}\;\mod{\psi(z)}T_\varphi(z)\\
\sigma_{0,\psi,\varphi} &= \sup_{z \in D}\;\omega_0(\varphi(z))Q_\psi(z)\\
\tau_{0,\psi,\varphi} &= \sup_{z \in D}\;\mod{\psi(z)}T_{0,\varphi(z)}.\end{aligned}$$  In the case of the unit disk, $\sigma_{\psi,\varphi} = \sigma_{0,\psi,\varphi}$, $\tau_{\psi,\varphi} = \tau_{0,\psi,\varphi}$, and these quantities agree with the expressions in the previous section.

\begin{theorem}\label{boundest} Let $D$ be a bounded homogeneous domain in $\C^n$ and $\varphi$ a holomorphic self-map of $D$.  If $\psi \in \Bloch(D)$, and $\sigma_{\psi,\varphi}$ and $\tau_{\psi,\varphi}$ are finite, then $W_{\psi,\varphi}$ is bounded on $\Bloch(D)$ and  {\small{$$\max\{\blochnorm{\psi}, \mod{\psi
(0)}\omega(\varphi(0))\} \leq \norm{W_{\psi,\varphi}} \leq \max\{\blochnorm{\psi}, \mod{\psi(0)}\omega(\varphi(0)) + \tau_{\psi,\varphi} + \sigma_{\psi,\varphi}\}.$$}}\end{theorem}

\begin{proof} We begin by proving the upper estimate.  Let $f \in \Bloch(D)$.  Then for $z \in D$, by the product rule we have $$\nabla(\psi(f\circ\varphi))(z) = \psi(z) \nabla(f\circ\varphi)(z) + f(\varphi(z))\nabla(\psi)(z),$$ so for all $u \in \C^n\setminus\{0\}$, $$\mod{\nabla(\psi(f\circ \varphi))(z)u} \leq \mod{\psi(z)}\mod{\nabla(f)(\varphi(z))J\varphi(z)u} + \mod{f(\varphi(z))}\mod{\nabla(\psi)(z)u}.$$  By (\ref{testimate}) and Lemma \ref{inequality:f(z) bound bloch}, we obtain $$\sup_{z \in D} Q_{\psi(f\circ\varphi)}(z) \leq \tau_{\psi,\varphi}\beta_f + \mod{f(0)}\beta_\psi + \sup_{z \in D}\omega(\varphi(z))Q_\psi(z)\beta_f,$$ which is finite.  So $W_{\psi,\varphi} f \in \Bloch(D)$ and again by Lemma \ref{inequality:f(z) bound bloch}
$$\begin{aligned}
\blochnorm{W_{\psi,\varphi} f} &\leq \mod{\psi(0)}\mod{f(\varphi(0))} + \mod{f(0)}\beta_\psi + (\tau_{\psi,\varphi} + \sigma_{\psi,\varphi})\beta_f\\
&\leq \mod{\psi(0)}(\mod{f(0)} + \omega(\varphi(0))\beta_f) + \mod{f(0)}\beta_\psi + (\tau_{\psi,\varphi} + \sigma_{\psi,\varphi})\beta_f\\
&= \blochnorm{\psi}\blochnorm{f} + (|\psi(0)|\omega(\varphi(0)) + \tau_{\psi,\varphi} + \sigma_{\psi,\varphi} - \blochnorm{\psi})\beta_f.
\end{aligned}$$  If $\mod{\psi(0)}\omega(\varphi(0)) + \tau_{\psi,\varphi} + \sigma_{\psi,\varphi} \leq \blochnorm{\psi}$, then $\blochnorm{W_{\psi,\varphi}f} \leq \blochnorm{\psi}\blochnorm{f}$.  Otherwise, $\blochnorm{W_{\psi,\varphi} f} \leq (\mod{\psi(0)}\omega(\varphi(0)) + \tau_{\psi,\varphi} + \sigma_{\psi,\varphi})\blochnorm{f}.$  Thus, $W_{\psi,\varphi}$ is bounded and $$\blochnorm{W_{\psi,\varphi}} \leq \max\{\blochnorm{\psi}, \mod{\psi(0)}\omega(\varphi(0)) + \tau_{\psi,\varphi} + \sigma_{\psi,\varphi}\}.$$

To prove the lower estimate, observe that by considering as test function the constant function 1, we have $\blochnorm{W_{\psi,\varphi} 1} = \blochnorm{\psi}$, so that $\norm{W_{\psi,\varphi}} \geq \blochnorm{\psi}$.  Furthermore
$$\begin{aligned}
\norm{W_{\psi,\varphi}} &= \sup\{\blochnorm{W_{\psi,\varphi}f} : f \in \Bloch(D) \text{ and } \blochnorm{f} \leq 1\}\\
&\geq \sup\{\blochnorm{W_{\psi,\varphi} f} : f \in \Bloch(D), f(0) = 0, \text{ and } \blochnorm{f} \leq 1\}\\
&\geq \sup\{\mod{\psi(0)}\mod{f(\varphi(0))} : f \in \Bloch(D), f(0) = 0, \text{ and } \blochnorm{f} \leq 1\}\\
&= \mod{\psi(0)}\omega(\varphi(0)).
\end{aligned}$$  Thus $\norm{W_{\psi,\varphi}} \geq \max\{\blochnorm{\psi}, \mod{\psi(0)}\omega(\varphi(0))\}.$
\end{proof}

\begin{theorem}\label{theorem:little bloch bounded} Let $D$ be a bounded homogeneous domain in $\C^n$.  If $\psi \in \Bloch_{0^*}(D)$, $\sigma_{\psi,\varphi}$ and $\tau_{\psi,\varphi}$ are finite, and $$\lim_{z \to \partial^*\hskip-2pt D} \mod{\psi(z)}T_{0,\varphi}(z) = \lim_{z \to \partial^*\hskip-2pt D} \omega_0(\varphi(z))Q_\psi(z) = 0,$$ then $W_{\psi,\varphi}$ is bounded on $\Bloch_{0^*}(D)$ and {\small{$$\max\{\blochnorm{\psi}, \mod{\psi
(0)}\omega_0(\varphi(0))\} \leq \norm{W_{\psi,\varphi}} \leq \max\{\blochnorm{\psi}, \mod{\psi(0)}\omega_0(\varphi(0)) + \tau_{\psi,\varphi} + \sigma_{0,\psi,\varphi}\}.$$}}\end{theorem}

\begin{proof} Arguing as in the proof of Theorem \ref{boundest}, it suffices to show that if $\psi \in \Bloch_{0^*}(D)$, $\sigma_{\psi,\varphi}$ and $\tau_{\psi,\varphi}$ are finite, and $$\lim_{z \to \partial^*\hskip-2pt D} \mod{\psi(z)}T_{0,\varphi}(z) = \lim_{z \to \partial^*\hskip-2pt D} \omega_0(\varphi(z))Q_\psi(z) = 0,$$ then $W_{\psi,\varphi}$ maps the $*$-little Bloch space into itself.  Let $f \in \Bloch_{0^*}(D)$.  Without loss of generality, we may assume $\blochnorm{f} \leq 1$.  For $z \in D$, by Lemma \ref{inequality:f(z) bound bloch}, we have
$$\begin{aligned}
Q_{\psi(f\circ\varphi)}(z) &\leq \mod{\psi(z)}Q_{f\circ\varphi}(z) + \mod{f(\varphi(z))}Q_\psi(z)\\
&\leq \mod{\psi(z)}T_{0,\varphi}(z) + \mod{f(0)}Q_\psi(z) + \omega_0(\varphi(z))Q_\psi(z),
\end{aligned}$$ which approaches 0 as $z \to \partial^*\hskip-2pt D$.  Thus $\psi(f\circ\varphi) \in \Bloch_{0^*}(D)$.
\end{proof}

\begin{theorem}\label{theorem:tau sigma finite} Let $D$ be a bounded homogeneous domain in $\C^n$, $\psi$ a holomorphic function on $D$, and $\varphi$ a holomorphic self-map of $D$.  If $W_{\psi,\varphi}$ is bounded on the Bloch space of $D$, then $\psi \in \Bloch(D)$ and $\sigma_{\psi,\varphi}$ is finite if and only if $\tau_{\psi,\varphi}$ is finite.\end{theorem}

\begin{proof} First observe that $\psi = W_{\psi,\varphi}1 \in \Bloch(D)$.  Let $f\in\Bloch(D)$, $z \in D$ and $u \in \C^n\setminus\{0\}$.  Then
$$\begin{aligned}
\frac{\mod{f(\varphi(z))}\mod{\nabla(\psi)(z)u}}{H_z(u,\conj{u})^{1/2}} &= \frac{\mod{\nabla(\psi(f\circ\varphi))(z)u - \psi(z)\nabla(f\circ\varphi)(z)u}}{H_z(u,\conj{u})^{1/2}}\\
&\leq \frac{\mod{\nabla(\psi(f\circ\varphi))(z)u}}{H_z(u,\conj{u})^{1/2}} + \frac{\mod{\psi(z)}{\mod{\nabla(f\circ\varphi)(z)u}}}{H_z(u,\conj{u})^{1/2}}
\end{aligned}$$  Taking the supremum over all $u \in \C^n\setminus\{0\}$, and using (\ref{testimate}) we get $$\begin{aligned}
\mod{f(\varphi(z))}Q_{\psi}(z) &\leq Q_{\psi(f\circ\varphi)}(z) + \mod{\psi(z)}Q_{f\circ\varphi}(z)\\
&\leq \beta_{\psi(f\circ\varphi)} + \mod{\psi(z)}T_\varphi(z)\beta_f\\
&\leq (\norm{W_{\psi,\varphi}} + \mod{\psi(z)}T_\varphi(z))\blochnorm{f}.
\end{aligned}$$  Taking the supremum over all $f \in \Bloch(D)$ with $f(0) = 0$ and $\blochnorm{f} \leq 1$, we have
$$\omega(\varphi(z))Q_\psi(z) \leq \norm{W_{\psi,\varphi}} + \mod{\psi(z)}T_\varphi(z).$$  Thus $\sigma_{\psi,\varphi} \leq \norm{W_{\psi,\varphi}} + \tau_{\psi,\varphi}.$

On the other hand, for $g\in\Bloch(D)$, with $g(0)=0$ and $\norm{g}_{\Bloch}\leq 1$, using Lemma~\ref{inequality:f(z) bound bloch}, we also obtain
$$\begin{aligned}
\mod{\psi(z)}Q_{g\circ\varphi}(z)&\leq Q_{\psi(g\circ\varphi)}(z)+\mod{g(\varphi(z))}Q_\psi(z)\\
&\leq \norm{W_{\psi,\varphi}g}_\Bloch+\omega(\varphi(z))Q_\psi(z)\\
&\leq \norm{W_{\psi,\varphi}}+\sigma_{\psi,\varphi}.\end{aligned}$$
More generally, for any non-constant function $f\in\Bloch(D)$, with $\beta_f\leq 1$, letting $g=(f-f(0))/\beta_f$, by the previous case, we obtain
$$\mod{\psi(z)}Q_{f\circ\varphi}(z)=\mod{\psi(z)}Q_{g\circ\varphi}(z)\beta_f\leq \norm{W_{\psi,\varphi}} + \sigma_{\psi,\varphi}.$$  Taking the supremum over all such functions $f$, we deduce $\tau_{\psi,\varphi}\leq \norm{W_{\psi,\varphi}} + \sigma_{\psi,\varphi}.$ Consequently, $\sigma_{\psi,\varphi}$ is finite if and only if $\tau_{\psi,\varphi}$ is finite.
\end{proof}

The proof of the following result is analogous.

\begin{proposition} \label{proposition:tau sigma finite} Let $D$ be a bounded homogeneous domain in $\C^n$, $\psi$ a holomorphic function on $D$, and $\varphi$ a holomorphic self-map of $D$.  If $W_{\psi,\varphi}$ is bounded on the $*$-little Bloch space of $D$, then $\psi \in \Bloch_{0^*}(D)$ and $\sigma_{0,\psi,\varphi}$ is finite if and only if $\tau_{0,\psi,\varphi}$ is finite.\end{proposition}

	We shall now give a sufficient condition for the compactness of $W_{\psi,\varphi}$ which yields Theorem~3 of \cite{ShiLuo:00} in the special case when $\psi$ is identically one.  We first need the following result.

\begin{lemma}\label{suffcond} Let $D$ be a bounded homogeneous domain in $\C^n$, $\psi$ a holomorphic function on $D$, and $\varphi$ a holomorphic self-map of $D$. Then $W_{\psi,\varphi}$ is compact on $\Bloch(D)$ if and only if for each bounded sequence $\{f_k\}$ in $\Bloch(D)$ converging to 0 locally uniformly in $D$, $\norm{\psi(f_k\circ \varphi)}_\Bloch\to 0$, as $k\to \infty$.\end{lemma}

\begin{proof}  Assume $W_{\psi,\varphi}$ is compact on $\Bloch(D)$.  Let $\{f_k\}$ be a bounded sequence in $\Bloch(D)$ which converges to 0 locally uniformly in $D$.  By rescaling $f_k$, we may assume $\blochnorm{f_k} \leq 1$ for all $k \in \N$.  We need to show that $\blochnorm{\psi(f_k\circ\varphi)} \to 0$ as $k \to \infty$.  Since $W_{\psi,\varphi}$ is compact, the sequence $\{\psi(f_k\circ\varphi)\}$ has a subsequence (which for convenience we reindex as the original sequence) which converges in the Bloch norm to some function $f \in \Bloch(D)$.  We are going to show that $f$ is identically 0 by proving that $\psi(f_k\circ\varphi) \to 0$ locally uniformly.  Fix $z_0 \in D$ and, without loss of generality, assume $f(z_0) = 0$.  Then $\psi(z_0)f_k(\varphi(z_0)) \to 0$ as $k \to \infty$.  For $z \in D$, by (\ref{|f(z)-f(w)| < rho}), we obtain
$$\begin{aligned}
\mod{\psi(z)f_k(\varphi(z)) - f(z)} &\leq \mod{\psi(z)f_k(\varphi(z)) - f(z) - (\psi(z_0)f_k(\varphi(z_0)) - f(z_0))}\\
 &\qquad + \mod{\psi(z_0)f_k(\varphi(z_0))}\\
&\leq \blochnorm{\psi(f_k\circ\varphi) - f}\rho(z,z_0) + \mod{\psi(z_0)f_k(\varphi(z_0))} \to 0
\end{aligned}$$ locally uniformly as $k \to \infty$, since $\psi(f_k\circ\varphi)-f \to 0$ in norm.  On the other hand, $\psi(f_k\circ\varphi) \to 0$ locally uniformly, so $f$ must be identically 0.

Next, assume $\blochnorm{\psi(g_n\circ\varphi)} \to 0$ as $k \to \infty$ for each bounded sequence $\{g_k\}$ in $\Bloch(D)$ converging to 0 locally uniformly in $D$.  To prove the compactness of $W_{\psi,\varphi}$, it suffices to show that if $\{f_k\}$ is a sequence in $\Bloch(D)$ with
$\norm{f_k}_\Bloch\leq 1$ for all $k\in\N$, there exists a subsequence $\{f_{k_j}\}$ such that $\psi(f_{k_j}\circ \varphi)$ converges in $\Bloch(D)$. Fix $z_0\in D$. Replacing $f_k$ with $f_k-f_k(z_0)$, we may assume that $f_k(z_0)=0$ for all $k\in\N$. By (\ref{beta_f alternative}), $\mod{f_k(z)}\le \rho(z,z_0)$, for each $z\in D$. Thus, on each closed ball centered at $z_0$ with respect to the Bergman distance, the sequence $\{f_k\}$ is uniformly bounded, and hence also on each compact subset of $D$. By Montel's theorem, some subsequence $\{f_{k_j}\}$ converges locally uniformly to some function $f$ holomorphic on $D$. By Theorem~3.3 of \cite{AllenColonna:07}, $f$ is a Bloch function and $\norm{f}_\Bloch\leq 1$. Then, letting $g_{k_j}=f_{k_j}-f$, we obtain a bounded sequence in $\Bloch(D)$ converging to 0 locally uniformly in $D$. Thus, by the hypothesis, $\norm{\psi(g_{k_j}\circ\varphi)}_\Bloch\to 0$ as $k\to\infty$. Therefore, $\psi(f_{k_j}\circ \varphi)$ converges in norm to $\psi(f\circ \varphi)$, completing the proof.
\end{proof}

\begin{theorem}\label{bhdcomp} Let $D$ be a bounded homogeneous domain in $\C^n$, $\psi$ a holomorphic function on $D$, and $\varphi$ a holomorphic self-map of $D$.  If $\psi\in \Bloch(D)$, then $W_{\psi,\varphi}$ is compact on the Bloch space of $D$ if
\begin{eqnarray}
\lim_{\varphi(z)\to \partial D}\omega(\varphi(z))Q_\psi(z)= 0 \text{ and } \lim_{\varphi(z)\to \partial D}\mod{\psi(z)}T_{\varphi}(z)=0.\label{threecond}\end{eqnarray}
\end{theorem}

\begin{proof} Assume the conditions in (\ref{threecond}) hold.  By Lemma \ref{suffcond}, to prove that $W_{\psi,\varphi}$ is compact on $\Bloch(D)$ it suffices to show that for any sequence $\{f_k\}$ in $\Bloch(D)$ converging to 0 locally uniformly in $D$ such that $\norm{f_k}_{\Bloch}\leq 1$, $\norm{\psi(f_k\circ\varphi)}_\Bloch\to 0$ as $k\to\infty$. Let $\{f_k\}$ be such a sequence and fix $\epsilon>0$.  Then $\mod{f_k(0)}<\epsilon/(3\norm{\psi}_\Bloch)$ for all $k$ sufficiently large and there exists $r$ such that for all $k\in \N$, $\mod{\psi(z)}Q_{f_k\circ\varphi}(z)<\epsilon/3$ and $\omega(\varphi(z))Q_\psi(z)<\epsilon/3$, whenever $\rho(\varphi(z),\partial D)<r$. Thus by Lemma~\ref{inequality:f(z) bound bloch}, if $\rho(\varphi(z),\partial D)<r$, then
\begin{eqnarray} Q_{\psi(f_k\circ\varphi)}(z)&\leq& \mod{\psi(z)}Q_{f_k\circ\varphi}(z)+\mod{f_k(\varphi(z))}Q_\psi(z)\nonumber\\
&<&\frac{\epsilon}{3}+(\mod{f_k(0)}+\omega(\varphi(z)))Q_\psi(z)<\epsilon.\nonumber\end{eqnarray}
On the other hand, since $f_k\to 0$ locally uniformly in $D$, $\mod{f_k(\varphi(z))}\to 0$ and $Q_{f_k\circ \varphi}\to 0$ uniformly on the set $\{z\in D: \rho(\varphi(z),\partial D)\ge r\}$. Consequently, for all $k$ sufficiently large, $Q_{\psi(f_k\circ \varphi)}(z)<\epsilon$ for all $z\in D$. Furthermore, $\mod{\psi(0)f_k(\varphi(0))}\to 0$ as $k\to\infty$, so $\norm{\psi(f_k\circ \varphi)}_\Bloch\to 0$, completing the proof.
\end{proof}

\begin{remark} Even for composition operators, the necessity of the analogue to Theorem \ref{bhdcomp} was established for the unit ball and polydisk \cite{ShiLuo:00}, but not for general bounded homogeneous domains.
\end{remark}

We end the section with the following conjecture.

\begin{conjecture} Let $D$ be a bounded homogeneous domain in $\C^n$, $\psi$ a holomorphic function on $D$, and $\varphi$ a holomorphic self-map of $D$.  Then $W_{\psi,\varphi}$ is bounded on the Bloch space of $D$ if and only if $\psi \in \Bloch(D)$, and $\sigma_{\psi,\varphi}$ and $\tau_{\psi,\varphi}$ are finite. Furthermore, the bounded operator $W_{\psi,\varphi}$ is compact on $\Bloch(D)$ if and only if $$\lim_{\varphi(z)\to \partial D}\;\omega(\varphi(z))Q_\psi(z) = \lim_{\varphi(z)\to \partial D}\;\mod{\psi(z)}T_\varphi(z) = 0.$$\end{conjecture}

In the next two sections, we prove the above conjecture when $D$ is the unit ball or the unit polydisk.

\section{Special Case: The Unit Ball}
In Theorem 3.1 of \cite{Zhu:04}, the following useful formula for calculating the Bloch semi-norm of a function $f \in \Bloch(\B_n)$ was given. For $z \in \B_n$ \begin{equation}\label{Q_f_zhu}Q_f(z) = (1-\norm{z}^2)^{1/2}\left(\norm{\nabla(f)(z)}^2 - \mod{\sum_{j=1}^n z_j\frac{\partial f}{\partial z_j}(z)}^2\right)^{1/2}.\end{equation}  Zhou and Chen characterized the bounded and the compact weighted composition operators on the Bloch space of the unit ball under the norm
\begin{equation}\label{normball2} \mod{f(0)} + \sup_{z \in \B_n} (1-\norm{z}^2)\norm{\nabla f(z)},\end{equation} which is equivalent to the Bloch norm on $\B_n$ \cite{Timoney:80-I}.  The following theorem is a special case of Corollaries 1.4 and 1.6 of \cite{ZhouChen:05}; their results apply to a large set of function spaces which includes the Bloch space.

\begin{theorem}[\cite{ZhouChen:05}]\label{thm:ZC} Let $\psi$ be a holomorphic function of $\B_n$ and $\varphi$ a holomorphic self-map of $\B_n$.  Then
$W_{\psi,\varphi}$ is bounded on $\Bloch(\B_n)$ if and only if
$$\sup_{z \in \B_n}\;\mod{\psi(z)}B_\varphi(z) < \infty,\hbox{ and }\sup_{z \in \B_n}\;(1-\norm{z}^2)\norm{\nabla \psi(z)}\log\frac{2}{1-\norm{\varphi(z)}^2} < \infty.$$  Furthermore, $W_{\psi,\varphi}$ is compact if and only if
\begin{eqnarray}&\lim\limits_{\norm{\varphi(z)}\to 1}\mod{\psi(z)}B_\varphi(z) = 0, \hbox{ and} \nonumber\\
\lim\limits_{\norm{\varphi(z)}\to 1}\hskip-7pt&(1-\norm{z}^2)\norm{\nabla \psi(z)}\displaystyle\log\frac{2}{1-\norm{\varphi(z)}^2} = 0.\label{compactcondition2}\end{eqnarray}
\end{theorem}

	We now show that the bounded and compact weighted composition operators can also be characterized in terms of the quantities $\sigma_{\psi,\varphi}$ and $\tau_{\psi,\varphi}$.

\begin{theorem}\label{theorem:wco bounded/compact B_n} Let $\psi$ be a holomorphic function on $\B_n$ and $\varphi$ a holomorphic self-map of $\B_n$.  Then
\begin{enumerate}
  \item [(a)] $W_{\psi,\varphi}$ is bounded on $\Bloch(\B_n)$ if and only if $\psi\in \Bloch(\B_n)$, and $\sigma_{\psi,\varphi}$ and $\tau_{\psi,\varphi}$ are finite.  \vskip4pt

\item [(b)] The bounded operator $W_{\psi,\varphi}$ is compact on $\Bloch(\B_n)$ if and only if
\begin{eqnarray}&\lim\limits_{\norm{\varphi(z)}\to 1}\mod{\psi(z)}T_\varphi(z) = 0, \hbox{ and }\nonumber\\
\label{compactcondition1} &\lim\limits_{\norm{\varphi(z)}\to 1}Q_\psi(z)\log\frac{1+\norm{\varphi(z)}}{1-\norm{\varphi(z)}} = 0.\end{eqnarray}
\end{enumerate}
\end{theorem}

\begin{remark}\label{remarkcompact} At first glance it may seem evident that conditions $(\ref{compactcondition2})$ and $(\ref{compactcondition1})$ are equivalent due to the equivalence between the norm $(\ref{normball2})$ and the Bloch norm. However, we have not been able to prove directly that
$(\ref{compactcondition2})$ implies $(\ref{compactcondition1})$ and thus, the proof of $(\ref{compactcondition1})$ under the compactness assumption does not make use of $(\ref{compactcondition2})$. \end{remark}

\begin{proof} (a)
If $\psi \in \Bloch(\B_n)$ and $\sigma_{\psi,\varphi}$ and $\tau_{\psi,\varphi}$ are finite, then $W_{\psi,\varphi}$ is bounded by Theorem \ref{boundest}. Conversely, assume $W_{\psi,\varphi}$ is bounded. Then $\psi= W_{\psi,\varphi}1  \in \Bloch(\B_n)$ and by Theorem~\ref{thm:ZC}, $\displaystyle\sup_{z \in \B_n} \mod{\psi(z)}B_\varphi(z)$ is finite. From (\ref{T-B inequality}), we deduce $$\tau_{\psi,\varphi} = \sup_{z \in \B_n}\mod{\psi(z)}T_\varphi(z) \leq \sup_{z \in \B_n}\mod{\psi(z)}B_\varphi(z) < \infty.$$ On the other hand, using Theorem \ref{theorem:tau sigma finite}, we see that $\sigma_{\psi,\varphi}$ is also finite, completing the proof of (a).

To prove (b) observe that by Theorem \ref{bhdcomp}, if $$\lim_{\norm{\varphi(z)}\to 1}\;Q_\psi(z)\log\frac{1+\norm{\varphi(z)}}{1-\norm{\varphi(z)}} = \lim_{\norm{\varphi(z)}\to 1}\;\mod{\psi(z)}T_\varphi(z) = 0,$$ then $W_{\psi,\varphi}$ is compact.
Conversely, assume $W_{\psi,\varphi}$ is compact. Then, from Theorem \ref{thm:ZC} we get $$\lim_{\norm{\varphi(z)}\to 1} \mod{\psi(z)}T_\varphi(z) \leq \lim_{\norm{\varphi(z)}\to 1} \mod{\psi(z)}B_\varphi(z) = 0.$$ Furthermore, $W_{\psi,\varphi}$ is bounded and so $$\sup_{z\in \B_n}Q_\psi(z)\log\frac{1+\norm{\varphi(z)}}{1-\norm{\varphi(z)}}<\infty.$$ In particular,
\begin{eqnarray} \lim_{\norm{\varphi(z)}\to 1}Q_\psi(z)=0.\label{Qzerolimit}\end{eqnarray}
Let $\{z_k\}$ be a sequence in $\B_n$ such that $\norm{\varphi(z_k)} \to 1$ as $k \to \infty$.  For $z \in \B_n$ define
\begin{equation}\nonumber f_k(z) = \frac{\left(\Log\displaystyle\frac{2}{1-\ip{z,\varphi(z_k)}}\right)^2}{\log\displaystyle\frac{2}{1-\norm{\varphi(z_k)}^2}}.\end{equation} Then $\{f_k\}$ converges to 0 locally uniformly in $\B_n$.  We are now going to show that $\{f_k\}$ is bounded in $\Bloch(\B_n)$. For $z \in \B_n$, set $$g_k(z) = \Log\frac{2}{1-\ip{z,\varphi(z_k)}}.
$$ Then by (\ref{Q_f_zhu}) and the Cauchy-Schwarz inequality, we have
\begin{eqnarray}
Q_{g_k}(z)\nonumber &=& (1-\norm{z}^2)^{1/2}\frac{\left(\norm{\varphi(z_k)}^2 - \mod{\ip{z,\varphi(z_k)}}^2\right)^{1/2}}{\mod{1-\ip{z,\varphi(z_k)}}}\\
&\leq &\sqrt{2}\frac{(1-\norm{z})^{1/2}}{(1-\mod{\ip{z,\varphi(z_k)}})^{1/2}}(1+\mod{\ip{z,\varphi(z_k)}})^{1/2}\leq 2.\nonumber\end{eqnarray}
Next, observe that for $z \in \B_n$
$$\nabla(f_k)(z) = \frac{2\Log\frac{2}{1-\ip{z,\varphi(z_k)}}}{\log\frac{2}{1-\norm{\varphi(z_k)}^2}}\nabla(g_k)(z).$$ So for $u \in \C^n\setminus\{0\}$
$$\begin{aligned}
\frac{\mod{\nabla(f_k)(z)u}}{H_z(u,\conj{u})^{1/2}} &\leq \frac{2\left(\log\frac{2}{1-\mod{\ip{z,\varphi(z_k)}}} + \frac{\pi}{2}\right)}{\log\left(\frac{2}{1-\norm{\varphi(z_k)}^2}\right)}\frac{\mod{\nabla(g_k)(z)u}}{H_z(u,\conj{u})^{1/2}}\\
&\leq \frac{2\left(\log\left(\frac{4}{1-\norm{\varphi(z_k)}^2}\right) + \frac{\pi}{2}\right)}{\log\frac{2}{1-\norm{\varphi(z_k)}^2}}Q_{g_k}(z)\leq 4\left(2+\frac{\pi}{2\log 2}\right).
\end{aligned}$$  Hence $\blochnorm{f_k}$ is bounded above by $\log 2 + 4\left(2+\frac{\pi}{2\log 2}\right)$.
By the compactness of $W_{\psi,\varphi}$, $\blochnorm{\psi(f_k\circ\varphi)} \to 0$ as $k \to \infty$.  Moreover $$\nabla (f_k)(\varphi(z_k)) = \frac{2\conj{\varphi(z_k)}}{1-\norm{\varphi(z_k)}^2},$$ so, for $u \in \C^n\setminus\{0\}$, we have
$$\mod{\nabla(f_k)(\varphi(z_k)) J\varphi(z_k)u} =
  \frac{2\mod{\ip{J\varphi(z_k)u, \varphi(z_k)}}}{1-\norm{\varphi(z_k)}^2}.$$
Hence
{{\begin{eqnarray}
\blochnorm{\psi(f_k\circ\varphi)} &\geq& \sup_{z \in \B_n} Q_{\psi(f_k\circ\varphi)}(z) \geq Q_{\psi(f_k\circ\varphi)}(z_k)\notag\\
&\geq& \mod{Q_\psi(z_k)f_k(\varphi(z_k)) - \mod{\psi(z_k)}\sup_{u \in \C^n\setminus\{0\}}\frac{\mod{\nabla(f_k)(\varphi(z_k))J\varphi(z_k)u}}{H_z(u,\conj{u})^{1/2}}}\notag\\
&=& {\text{\Huge{$|$}}}Q_\psi(z_k)\log\frac{2}{1-\norm{\varphi(z_k)}^2}\notag\\
&& \qquad -\, \frac{2\mod{\psi(z_k)}}{1-\norm{\varphi(z_k)}^2}\sup_{u \in \C^n\setminus\{0\}}\frac{\mod{\ip{J\varphi(z_k)u,\varphi(z_k)}}}{H_z(u,\conj{u})^{1/2}}{\text{\Huge{$|$}}}.\notag
\end{eqnarray}}}

\noindent We now show that \begin{eqnarray}\lim_{k \to \infty} \frac{\mod{\psi(z_k)}}{1-\norm{\varphi(z_k)}^2}\sup_{u \in \C^n\setminus\{0\}}\frac{\mod{\ip{J\varphi(z_k)u,\varphi(z_k)}}}{H_z(u,\conj{u})^{1/2}} = 0.\label{toshow}\end{eqnarray}  Once this is proved, it will follow that $$\lim_{k \to \infty} Q_\psi(z_k)\log\frac{2}{1-\norm{\varphi(z_k)}^2} = 0$$ since $\blochnorm{\psi(f_k\circ\varphi)} \to 0$ as $k \to \infty$.  Noting that
\begin{eqnarray} Q_\psi(z_k)\log\frac{2}{1-\norm{\varphi(z_k)}^2}&\geq& Q_\psi(z_k)\log\frac1{2}\,\frac{1+\norm{\varphi(z_k)}}{1-\norm{\varphi(z_k)}}\nonumber\\
&=&Q_\psi(z_k)\log\frac{1+\norm{\varphi(z_k)}}{1-\norm{\varphi(z_k)}}-Q_\psi(z_k)\log 2\nonumber\end{eqnarray}
and that by (\ref{Qzerolimit}), $\lim\limits_{k\to \infty}Q_\psi(z_k)=0$, we obtain that the limit of the first term of the above difference also goes to 0 as $k\to\infty$, and hence
$$\lim_{\norm{\varphi(z)} \to 1} Q_\psi(z)\log\frac{1+\norm{\varphi(z)}}{1-\norm{\varphi(z)}} = 0.$$

Let us now proceed with the proof of (\ref{toshow}). For $k \in \N$ and $z \in \B_n$, let $$h_k(z) = \frac{1-\norm{\varphi(z_k)}^2}{1-\ip{z,\varphi(z_k)}}.$$  Then $h_k \to 0$ uniformly on compact subsets of $\B_n$, and for $j = 1, \dots, n$ and $z\in\B_n$
$$\frac{\partial h_k}{\partial z_j}(z) = \frac{(1-\norm{\varphi(z_k)}^2)\conj{\varphi_j(z_k)}}{(1-\ip{z,\varphi(z_k)})^2}.$$  Thus by (\ref{Q_f_zhu}), we obtain
$$\begin{aligned}
Q_{h_k}(z) &= \frac{(1-\norm{z}^2)^{1/2}(1-\norm{\varphi(z_k)}^2)}{\mod{1-\ip{z,\varphi(z_k)}}^2}\left(\norm{\varphi(z_k)}^2 - \mod{\ip{z,\varphi(z_k)}}^2\right)^{1/2}\\
&\leq \frac{\sqrt{2}(1-\norm{z}^2)^{1/2}(1-\norm{\varphi(z_k)}^2)}{(1-\mod{\ip{z,\varphi(z_k)}})^{3/2}}\\
&\leq \frac{\sqrt{2}(1-\norm{z}^2)^{1/2}(1-\norm{\varphi(z_k)}^2)}{(1-\norm{z})^{1/2}(1-\norm{\varphi(z_k)})} \leq 4.
\end{aligned}$$  So $h_k \in \Bloch(\B_n)$ and $\blochnorm{h_k} \leq 5$.  Since $W_{\psi,\varphi}$ is compact, $\blochnorm{\psi(h_k\circ\varphi)} \to 0$ as $k \to \infty$.  Moreover
$$\begin{aligned}
\blochnorm{\psi(h_k\circ\varphi)} &\geq Q_{\psi(h_k\circ\varphi)}(z_k)\\
&\geq \mod{Q_\psi(z_k) - \mod{\psi(z_k)}\sup_{u \in \C^n\setminus\{0\}}\frac{\mod{\nabla(h_k)(\varphi(z_k))J\varphi(z_k)u}}{H_z(u,\conj{u})^{1/2}}}\\
&= \mod{Q_\psi(z_k) - \frac{\mod{\psi(z_k)}}{1-\norm{\varphi(z_k)}^2}\sup_{u \in \C^n\setminus\{0\}}\frac{\mod{\ip{J\varphi(z_k)u,\varphi(z_k)}}}{H_z(u,\conj{u})^{1/2}}}.
\end{aligned}$$  Since $\lim\limits_{k \to \infty} Q_\psi(z_k) = 0$, it follows that $$\lim_{k \to \infty} \frac{\mod{\psi(z_k)}}{1-\norm{\varphi(z_k)}^2}\sup_{u \in \C^n\setminus\{0\}}\frac{\mod{\ip{J\varphi(z_k)u,\varphi(z_k)}}}{H_z(u,\conj{u})^{1/2}} = 0.$$
The proof is now complete.
\end{proof}

	Next, we give an example of a bounded weighted composition operator on the unit ball whose associated component multiplication operator is unbounded and an example of a compact operator on $\B_n$ whose associated component operators are both not compact.
\smallskip

\noindent{\it{Examples}} (a) Let $\lambda \in \partial\B_n$ and define the functions $\psi(z) = \frac{1}{2}\hbox{Log}(1-\ip{z,\lambda})$ and $\varphi(z) = \frac{1}{2}(\lambda - z)$ for $z \in \B_n$.  The associated multiplication operator $M_\psi$ is not bounded on $\Bloch(\B_n)$ since $\psi \not\in H^\infty(\B_n)$. On the other hand, it is straightforward to verify that $\sup_{z \in \B_n}\mod{\psi(z)}B_\varphi(z) < \infty$ and $$
\sup_{z \in \B_n}(1-\norm{z}^2)\norm{\nabla\psi(z)}\log\frac{1}{1-\norm{\varphi(z)}^2} < \infty.$$  Therefore, $W_{\psi,\varphi}$ is bounded on $\Bloch(\B_n)$.
\smallskip

\noindent (b) Let $\psi(z) = 1-z$ and $\varphi(z) = \frac{1+z}{2}$ for $z \in \B_n$.  The multiplication operator $M_\psi$ is not compact on $\Bloch(\B_n)$, since $\psi$ is not identically zero. Moreover, the composition operator $C_\varphi$ is not compact on $\Bloch(\B_n)$ \cite{ShiLuo:00}, since
$$\begin{aligned}\frac{H_{\varphi(z)}(J\varphi(z)u,\conj{J\varphi(z)u})}{H_{z}(u,\conj{u})} &= \frac{1}{4}\frac{(1-\norm{\varphi(z)}^2)\norm{u}^2 + \mod{\ip{\varphi(z),u}}^2}{(1-\norm{z}^2)\norm{u}^2 + \mod{\ip{z,u}}^2}\frac{(1-\norm{z}^2)^2}{(1-\norm{\varphi(z)}^2)^2},\end{aligned}$$ which does not go to 0 if $z \to 1$ along the real axis in the first coordinate and $u = (1,0,\dots,0)$. Observe that
$$\lim_{\norm{\varphi(z)}\to 1}\psi(z)=\lim_{z\to 1}\psi(z)=0$$ and $B_\varphi(z)$ is bounded above by a constant independent of $\varphi$, so $$\lim_{\norm{\varphi(z)}\to 1} \mod{\psi(z)}B_\varphi(z) = 0.$$ Moreover, $$\lim_{\norm{\varphi(z)}\to 1} (1-\norm{z}^2)\log\frac{2}{1-\left\|\frac{1+z}{2}\right\|^2} = 0.
$$ Therefore $W_{\psi,\varphi}$ is compact on $\Bloch(\B_n)$.

\section{Special Case: The Unit Polydisk}
\begin{theorem}\label{theorem:wco bounded D^n} Let $\psi$ be a holomorphic function on $\D^n$ and $\varphi$ a holomorphic self-map of $\D^n$.  Then $W_{\psi,\varphi}$ is bounded on $\Bloch(\D^n)$ if and only if $\psi\in\Bloch(\D^n)$, and $\sigma_{\psi,\varphi}$ and $\tau_{\psi,\varphi}$ are finite. Furthermore, the bounded operator $W_{\psi,\varphi}$ is compact on $\Bloch(\D^n)$ if and only if \begin{eqnarray}\lim_{\varphi(z)\to \partial \D^n}\;\omega(\varphi(z))Q_\psi(z) = \lim_{\varphi(z)\to \partial \D^n}\;\mod{\psi(z)}T_\varphi(z) = 0.\label{sigmatauzeropoly}\end{eqnarray}
\end{theorem}

To prove this result, we will show that the conditions for the boundedness and compactness of $W_{\psi,\varphi}$ are equivalent to the conditions proven by Zhou and Chen in the following theorem.  Their results were obtained by considering on the Bloch space of $\D^n$ the norm $$\norm{f}_* = \mod{f(0)} + \sup_{z \in \D^n} \sum_{j=1}^n (1-\mod{z_j}^2)\mod{\frac{\partial f}{\partial z_j}(z)}.$$  In \cite{CohenColonna:08}, it was shown that for $f \in \Bloch(\D^n)$ and $z \in \D^n$, $$Q_f(z) = \norm{\left((1-\mod{z_1}^2)\frac{\partial f}{\partial z_1}(z), \cdots, (1-\mod{z_n}^2)\frac{\partial f}{\partial z_n}(z)\right)}.$$  Thus $\norm{\cdot}_*$ is equivalent to the Bloch norm since
\begin{equation}\label{equivnormpoly}\frac{1}{n}\sum_{j=1}^n (1-\mod{z_j}^2)\mod{\frac{\partial f}{\partial z_j}(z)} \leq Q_f(z) \leq \sum_{j=1}^n (1-\mod{z_j}^2)\mod{\frac{\partial f}{\partial z_j}(z)},\end{equation} for all $z \in \D^n$.

\begin{theorem}[\cite{ZhouChen-I:05}, Theorems 1 and 2]\label{thm:ZC-1} Let $\psi$ be a holomorphic function on $\D^n$ and $\varphi$ a holomorphic self-map of $\D^n$.  Then $W_{\psi,\varphi}$ is bounded on $\Bloch(\D^n)$ if and only if $$\sup_{z \in \D^n}\;\sum_{j,k=1}^n(1-\mod{z_j}^2)\mod{\frac{\partial \psi}{\partial z_j}(z)}\log\frac{4}{1-\mod{\varphi_k(z)}^2} < \infty$$ and $$\sup_{z \in \D^n}\;\mod{\psi(z)}\sum_{j,k = 1}^n\mod{\frac{\partial\varphi_k}{\partial z_j}(z)}\frac{1-\mod{z_j}^2}{1-\mod{\varphi_k(z)}^2} < \infty.$$  Furthermore, $W_{\psi,\varphi}$ is compact on $\Bloch(\D^n)$ if and only if $W_{\psi,\varphi}$ is bounded and $$\lim_{\varphi(z) \to \partial\D^n} \sum_{j,k= 1}^n(1-\mod{z_j}^2)\mod{\frac{\partial \psi}{\partial z_j}(z)}\log\frac{4}{1-\mod{\varphi_k(z)}^2} = 0$$ and $$\lim_{\varphi(z) \to \partial\D^n}\mod{\psi(z)}\sum_{j,k = 1}^n\mod{\frac{\partial\varphi_k}{\partial z_j}(z)}\frac{1-\mod{z_j}^2}{1-\mod{\varphi_k(z)}^2} = 0.$$
\end{theorem}

\begin{lemma}\label{greek} Let $\psi$ be a holomorphic function on $\D^n$ and $\varphi$ a holomorphic self-map of $\D^n$. Then, for $z \in \D^n$, the following inequalities hold:
\begin{enumerate}
  \item[(a)] $\rho(0,z)\leq \sum_{j=1}^n\log\frac4{1-\mod{z_j}^2}$;
  \item[(b)] $\sigma_{\psi,\varphi}(z)\leq \left(\sum_{j=1}^n (1-\mod{z_j}^2)\mod{\frac{\partial \psi}{\partial z_j}(z)}\right)\sum_{k=1}^n\log\frac4{1-\mod{\varphi_k(z)}^2}$;
  \item[(c)] $T_\varphi(z)\leq \sum_{j,k=1}^n\left|\frac{\partial \varphi_k}{\partial z_j}(z)\right|\frac{1-\mod{z_j}^2}{1-\mod{\varphi_k(z)}^2};$
  \item[(d)] $\tau_{\psi,\varphi}(z)\le \mod{\psi(z)}\sum_{j,k=1}^n\left|\frac{\partial \varphi_k}{\partial z_j}(z)\right|\frac{1-\mod{z_j}^2}{1-\mod{\varphi_k(z)}^2}$;
\end{enumerate}
\end{lemma}

\begin{proof} Let $z \in D^n$.  To prove (a), observe that for $u\in\C^n$, $$H_z(u,\overline{u})=\sum_{j=1}^n\frac{|u_j|^2}{(1-|z_j|^2)^2}$$ (e.g. see \cite{Timoney:80-I}), and recall that if $\gamma=\gamma(t)$ ($0\leq t\leq 1$) is the geodesic from $w$ to $z$, then $$\rho(w,z)=\int_0^1 H_{\gamma(t)}(\gamma'(t),\overline{\gamma'(t)})^{1/2}\,dt.$$ Since the geodesic from 0 to $z\in\D^n$ is parametrized by $\gamma(t)=tz$, for $0\leq t\leq 1$, we obtain
\begin{eqnarray}\label{polydistestimate}
\rho(0,z) &=& \int_0^1\left(\sum_{j=1}^n\frac{\mod{z_j}^2}{(1-\mod{z_j}^2t^2)^2}\right)^{1/2}\;dt \leq \int_0^1\sum_{j=1}^n \frac{\mod{z_j}}{1-\mod{z_j}^2t^2}\;dt\\ &=& \frac{1}{2}\sum_{j=1}^n\log\frac{1+\mod{z_j}}{1-\mod{z_j}},\nonumber
\end{eqnarray}
proving (a). By the upper estimate of (\ref{equivnormpoly}) and the inequality $\omega(\varphi(z))\leq \rho(0,z)$, part (b) follows immediately from part (a).

To prove (c), observe that by (1.2) of \cite{CohenColonna:08},
\begin{eqnarray} T_\varphi(z)&\leq& B_\varphi(z)=\max_{\norm{w}=1}\left(\sum_{k=1}^n\left|\sum_{j=1}^n\frac{\partial \varphi_k}{\partial z_j}(z)\frac{(1-\mod{z_j}^2)w_j}{1-\mod{\varphi_k(z)}^2}\right|^2\right)^{1/2}\nonumber\\
&\leq & \max_{\norm{w}=1}\left(\sum_{k=1}^n\left(\sum_{j=1}^n\left|\frac{\partial \varphi_k}{\partial z_j}(z)\right|\frac{(1-\mod{z_j}^2)\mod{w_j}}{1-\mod{\varphi_k(z)}^2}\right)^2\right)^{1/2}\nonumber\\
&\leq & \sum_{j,k=1}^n\left|\frac{\partial \varphi_k}{\partial z_j}(z)\right|\frac{1-\mod{z_j}^2}{1-\mod{\varphi_k(z)}^2}.\nonumber\end{eqnarray}

Part (d) is an immediate consequence of the formula $\tau_{\psi,\varphi}(z)=\mod{\psi(z)}T_{\varphi}(z)$ and part (c).\end{proof}

\begin{proof}[Proof of Theorem \ref{theorem:wco bounded D^n}.] If $W_{\psi,\varphi}$ is bounded, then $\psi=W_{\psi,\varphi}1\in\Bloch(\D^n)$, and from Theorem~\ref{thm:ZC-1}, inequality (d) of Lemma \ref{greek}, and Theorem \ref{theorem:tau sigma finite}, it follows that $\sigma_{\psi,\varphi}$ and $\tau_{\psi,\varphi}$ are finite. Conversely, if $\psi\in\Bloch(\D^n)$, and $\sigma_{\psi,\varphi}$ and $\tau_{\psi,\varphi}$ are finite, then $W_{\psi,\varphi}$ is bounded by Theorem \ref{boundest}.

	Next, assume $W_{\psi,\varphi}$ is compact. Then, by Theorem \ref{thm:ZC-1}, and inequalities (b) and (d) of Lemma \ref{greek}, the conditions in (\ref{sigmatauzeropoly}) hold. Conversely, if $W_{\psi,\varphi}$ is bounded and the conditions in (\ref{sigmatauzeropoly}) hold, then $W_{\psi,\varphi}$ is compact by Theorem \ref{bhdcomp}.
\end{proof}

	We conclude the section by giving an example of a bounded weighted composition operator on the polydisk whose corresponding component multiplication operator is not bounded, and an example of a compact weighted composition operator on $\D^n$ whose both component operators are not compact.
\smallskip

\noindent{\it{Examples.}} (a) Fix an index $j \in \{1, \dots, n\}$ and define $\psi(z) = \hbox{Log}\frac{2}{1-z_j}$ and $\varphi(z) = \frac{1-z_j}{2}$ for $z \in \D^n$.  Since $\psi \not\in H^\infty(\D^n)$, the associated multiplication operator $M_\psi$ is unbounded on $\Bloch(\D^n)$. On the other hand $$\begin{aligned}\sigma_{\psi,\varphi}\le \sup_{z \in \D^n}(1-\mod{z_j}^2)\log\frac{4}{1-\mod{\frac{1-z_j}{2}}^2} < \infty,\\
\tau_{\psi,\varphi}\leq \sup_{z \in \D^n}\frac{1-\mod{z_j}^2}{1-\mod{\frac{1-z_j}{2}}^2}\log\frac{2}{\mod{1-z_j}} < \infty,\end{aligned}$$
so that $W_{\psi,\varphi}$ is bounded on $\Bloch(\D^n)$.
\smallskip

\noindent (b) Fix an index $j \in \{1, \dots, n\}$, and for $z \in \D^n$ define $\psi(z) = 1-z_j$, and let $\varphi(z)$ be the vector with $k$th component 0 for $k\ne j$ and $j$th component $\frac{1+z_j}{2}$. Clearly, the associated multiplication operator $M_\psi$ is not compact on $\Bloch(\D^n)$. The associated composition operator $C_\varphi$ is not compact on $\Bloch(\D^n)$ since $$B_\varphi(z) = \frac12\frac{1-\mod{z_j}^2}{1-\mod{\frac{1+z_j}{2}}^2} \not\to 0$$ for $z_j\to 1$ \cite{ShiLuo:00}. Furthermore, $$\begin{aligned}\lim_{\varphi(z)\to \partial \D^n}\sigma_{\psi,\varphi}(z)\leq \lim_{z_j \to 1} (1-\mod{z_j}^2)\log\frac{4}{1-\mod{\frac{1+z_j}{2}}^2} &= 0,\hbox{ and }\\
\lim_{\varphi(z) \to \partial \D^n}\tau_{\psi,\varphi}(z)\leq \frac12 \lim_{z_j\to 1}\mod{1-z_j}\frac{1-\mod{z_j}^2}{1-\mod{\frac{1+z_j}{2}}^2} &= 0.\end{aligned}$$ Hence $W_{\psi,\varphi}$ is compact on $\Bloch(\D^n)$.

\section{Weighted composition operators from the Bloch spaces into $H^\infty$}
	In \cite{HosokawaIzuchiOhno:05}, Hosokawa, Izuchi and Ohno characterized the bounded and the compact weighted composition operators from $\Bloch(\D)$ and $\lilBloch(\D)$ into $H^\infty(\D)$. We now provide a characterization of the bounded operators in the environment of a bounded homogeneous domain and determine the operator norm. We also obtain an extension of their results when the domain is the unit ball and the unit polydisk.

\begin{theorem}\label{charaboundedness} Let $D$ be a bounded homogeneous domain, $\psi$ a holomorphic function on $D$, and $\varphi$ a holomorphic self-map of $D$. Then
\begin{enumerate}
\item[(a)] $W_{\psi,\varphi}:\Bloch(D)\to H^{\infty}(D)$ is bounded if and only if $\psi\in H^\infty(D)$ and
$$\eta_{\psi,\varphi}:=\sup_{z\in D}|\psi(z)|\omega(\varphi(z))<\infty.$$ If $W_{\psi,\varphi}$ is bounded on $\Bloch(D)$, then
\begin{equation}\label{estim} \norm{W_{\psi,\varphi}}=\max\{\norm{\psi}_\infty,\eta_{\psi,\varphi}\}.\end{equation}
\vskip4pt

\item[(b)] $W_{\psi,\varphi}:\Bloch_{0*}(D)\to H^{\infty}(D)$ is bounded if and only if $\psi\in H^\infty(D)$ and  $$\eta_{\,0,\psi,\varphi}:=\sup_{z\in D}|\psi(z)|\omega_0(\varphi(z))<\infty.$$ If $W_{\psi,\varphi}$ is bounded on $\Bloch_{0*}(D)$, then
\begin{equation}
\nonumber\norm{W_{\psi,\varphi}}=\max\{\norm{\psi}_\infty,\eta_{\,0,\psi,\varphi}\}.\end{equation}
\end{enumerate}
\end{theorem}

\begin{proof} To prove (a), assume $W_{\psi,\varphi}$ is bounded on $\Bloch(D)$. Then $\psi=W_{\psi,\varphi}1\in H^\infty(D)$, $\norm{\psi}_\infty\leq \norm{W_{\psi,\varphi}}$, and for each $f\in \Bloch(D)$ with $\norm{f}_\Bloch\leq 1$, and for each $z\in D$, we have
$$\norm{W_{\psi,\varphi}}\geq \norm{\psi(f\circ \varphi)}_{\infty}\geq |\psi(z)||f(\varphi(z))|.$$
Taking the supremum over all such functions $f$ such that $f(0)=0$, and over all $z\in D$, we obtain $\norm{W_{\psi,\varphi}}\geq\eta_{\psi,\varphi}$, proving that $\eta_{\psi,\varphi}<\infty$ and
\begin{equation}\label{lowerestim} \norm{W_{\psi,\varphi}}\geq \max\{\norm{\psi}_\infty, \eta_{\psi,\varphi}\}.\end{equation}

	Conversely, suppose $\psi\in H^\infty(D)$ and $\eta_{\psi,\varphi}$ is finite. Then, by Lemma~\ref{inequality:f(z) bound bloch}, for each $f\in \Bloch(D)$ we have
\begin{eqnarray} \sup_{z\in D}|\psi(z)||f(\varphi(z))|&\leq &\sup_{z\in D}|\psi(z)|(|f(0)|+\omega(\varphi(z))\beta_f)\nonumber\\
&\leq&\norm{\psi}_\infty(\norm{f}_\Bloch-\beta_f)+\eta_{\psi,\varphi}\beta_f\nonumber\\&\leq& \max\{\norm{\psi}_\infty,\eta_{\psi,\varphi}\}\norm{f}_\Bloch,\label{upperestim}\end{eqnarray}
proving the boundedness of $W_{\psi,\varphi}$. From (\ref{lowerestim}) and (\ref{upperestim}) we also obtain (\ref{estim}).
The proof of (b) is analogous.
\end{proof}

	Recalling that for each $z\in \B_n$, $$\omega_0(z)=\omega(z)=\frac12\log\frac{1+\norm{z}}{1-\norm{z}},$$ we deduce the following extension to the unit ball of Theorem~6.1 of \cite{HosokawaIzuchiOhno:05}, which is equivalent to Theorem 1 in \cite{LiStevic:08}. The evaluation of the operator norm has not appeared before.

\begin{corollary}\label{Ballcharaboundedness} Let $\psi$ be a holomorphic function on $\B_n$ and $\varphi$ a holomorphic self-map of $\B_n$. Then the following statements are equivalent:
\begin{enumerate}
\item[(a)] $W_{\psi,\varphi}:\Bloch(\B_n)\to H^{\infty}(\B_n)$ is bounded.

\item[(b)] $W_{\psi,\varphi}:\lilBloch(\B_n)\to H^{\infty}(\B_n)$ is bounded.

\item[(c)] $\psi\in H^\infty(\B_n)$ and  $$\sup_{z\in \B_n}|\psi(z)|\log\frac{1+\norm{\varphi(z)}}{1-\norm{\varphi(z)}}<\infty.$$
\end{enumerate}
Furthermore, $$\norm{W_{\psi,\varphi}}=\max\left\{\norm{\psi}_\infty, \sup_{z\in\B_n}\frac12|\psi(z)|\log\frac{1+\norm{\varphi(z)}}{1-\norm{\varphi(z)}}\right\}.$$
\end{corollary}

\noindent In the case where $\psi$ is the constant function one, the condition of the finiteness of $$\sup_{z\in \B_n}\log\frac{1+\norm{\varphi(z)}}{1-\norm{\varphi(z)}}$$ implies that $\varphi(z)$ cannot approach the boundary, or else the logarithmic term would tend to infinity.  Thus, we have the following corollary.

\begin{corollary} Let $\varphi$ be a holomorphic self-map of $\B_n$.  Then the following statements are equivalent:
\begin{enumerate}
\item[(a)] $C_\varphi:\Bloch(\B_n)\to H^{\infty}(\B_n)$ is bounded.

\item[(b)] $C_\varphi:\lilBloch(\B_n)\to H^{\infty}(\B_n)$ is bounded.

\item[(c)] $\varphi(\B_n)$ has compact closure in $\B_n$.
\end{enumerate}
Furthermore, the operator norm of $C_{\varphi}$ is the maximum between 1 and the Bergman distance of the boundary of the range of $\varphi$ from the origin.
\end{corollary}

We now show that Theorem~6.1 of \cite{HosokawaIzuchiOhno:05} can be extended to the unit polydisk.

\begin{theorem}\label{Polydiskcharaboundedness} Let $\psi$ be a holomorphic on $\D^n$ and $\varphi$ a holomorphic self-map of $\D^n$. Then the following statements are equivalent:
\begin{enumerate}
\item[(a)] $W_{\psi,\varphi}:\Bloch(\D^n)\to H^{\infty}(\D^n)$ is bounded.

\item[(b)] $W_{\psi,\varphi}:\Bloch_{0*}(\D^n)\to H^{\infty}(\D^n)$ is bounded.

\item[(c)] $\psi\in H^\infty(\D^n)$ and
\begin{eqnarray}\label{polybddcond}\sup_{z\in \D^n}|\psi(z)|\sum_{j=1}^n\log\frac{1+\mod{\varphi_j(z)}}{1-\mod{\varphi_j(z)}}<\infty.\end{eqnarray}
\end{enumerate}
\end{theorem}

\begin{proof} The implication $(a)\Longrightarrow (b)$ is obvious.

\noindent $(b)\Longrightarrow (c)$: It is clear that $\psi\in H^\infty(\D^n)$.
Fix $j=1,\dots,n$ and $\lambda\in \D^n$. For $z\in \D^n$ define
$$h(z)=\hbox{Log}\frac4{1-z_j\overline{\varphi_j(\lambda)}}.$$ Then $$\norm{h}_\Bloch =2\log 2+\sup_{|z_j|<1}\frac{(1-|z_j|^2)|\varphi_j(\lambda)|}{|1-z_j\overline{\varphi_j(\lambda)}|}\leq 2\log 2+2.$$ Furthermore,
$$Q_{h}(z)\leq \frac{(1-\mod{z_j}^2)|\varphi_j(\lambda)|}{1-|\varphi_j(\lambda)|}\to 0$$
as $\mod{z_j}\to 1$. Thus, $h\in \Bloch_{0*}(\D^n)$. By the boundedness of $W_{\psi,\varphi}:\Bloch_{0*}(\D^n)\to H^{\infty}(\D^n)$, we obtain
\begin{eqnarray} (2\log 2+2)\norm{W_{\psi,\varphi}}&\geq &\norm{\psi(h\circ \varphi)}_\infty\geq |\psi(\lambda)|\log\frac4{1-\mod{\varphi_j(\lambda)}^2}\nonumber\\
&\geq& |\psi(\lambda)|\log\frac{1+\mod{\varphi_j(\lambda)}}{1-\mod{\varphi_j(\lambda)}}.\nonumber\end{eqnarray}
Summing over all $j=1,\dots,n$ and taking the supremum over all $\lambda\in\D^n$, we get (\ref{polybddcond}).

\noindent $(c)\Longrightarrow (a)$: Observe that for $z\in\D^n$, by (\ref{polydistestimate}), we have
\begin{eqnarray}\label{omegarhobounds}\omega(\varphi(z))\leq \rho(0,\varphi(z))\leq \frac12\sum_{j=1}^n\log\frac{1+\mod{\varphi_j(z)}}{1-\mod{\varphi_j(z)}}.\end{eqnarray}
The result follows at once from Theorem~\ref{charaboundedness}(a).
\end{proof}

We now give a sufficient condition for compactness which can be proved as Theorem~\ref{bhdcomp}.

\begin{theorem}\label{suffcondinfty} Let $D$ be a bounded homogeneous domain, $\psi$ a holomorphic function on $D$, and $\varphi$ a holomorphic self-map of $D$. Then $W_{\psi,\varphi}:\Bloch(D)\to H^{\infty}(D)$ is compact if $\psi\in H^\infty(D)$ and
$$\lim_{\varphi(z)\to \partial D}|\psi(z)|\omega(\varphi(z))=0.$$
\end{theorem}

This sufficient condition is also necessary when $D$ is the unit ball. Indeed, the following result, proved by Li and Stevi\'c in \cite{LiStevic:08} (Theorem~4), is the extension to the unit ball of Theorem~6.2 of \cite{HosokawaIzuchiOhno:05}.

\begin{theorem}[\cite{LiStevic:08}] \label{B} Let $\psi$ be a holomorphic function on $\B_n$ and $\varphi$ a holomorphic self-map of $\B_n$. Then the following statements are equivalent:
\begin{enumerate}
\item[(a)] $W_{\psi,\varphi}:\Bloch(\B_n)\to H^{\infty}(\B_n)$ is compact.

\item[(b)] $W_{\psi,\varphi}:\lilBloch(\B_n)\to H^{\infty}(\B_n)$ is compact.

\item[(c)] $\psi\in H^\infty(\B_n)$ and  $$\lim_{\norm{\varphi(z)}\to 1}|\psi(z)|\log\frac{2}{1-\norm{\varphi(z)}}=0.$$
\end{enumerate}
\end{theorem}

The following is a consequence of Corollary~\ref{Ballcharaboundedness} and Theorem~\ref{B}. It follows immediately from the finiteness of $$\sup_{z \in \B_n}\mod{\psi(z)}\log\frac{1+\norm{z}}{1-\norm{z}},$$ which implies that $\psi$ is identically zero.

\begin{corollary} Let $\psi$ be a holomorphic function on $\B_n$.  Then the following are equivalent:
\begin{enumerate}
\item[(a)] $M_\psi:\Bloch(\B_n) \to H^\infty(\B_n)$ is bounded.
\item[(b)] $M_\psi:\lilBloch(\B_n) \to H^\infty(\B_n)$ is bounded.
\item[(c)] $M_\psi:\Bloch(\B_n) \to H^\infty(\B_n)$ is compact.
\item[(d)] $M_\psi:\lilBloch(\B_n) \to H^\infty(\B_n)$ is compact.
\item[(e)] $\psi$ is identically zero.
\end{enumerate}
\end{corollary}

We now prove the analogue of Theorem~\ref{B} for the polydisk.

\begin{theorem}\label{Polydiskcharacompactness}
Let $\psi$ be holomorphic on $\D^n$ and $\varphi$ a holomorphic self-map of $\D^n$. Then the following statements are equivalent:
\begin{enumerate}
\item[(a)] $W_{\psi,\varphi}:\Bloch(\D^n)\to H^{\infty}(\D^n)$ is compact.

\item[(b)] $W_{\psi,\varphi}:\Bloch_{0*}(\D^n)\to H^{\infty}(\D^n)$ is compact.

\item[(c)] $\psi\in H^\infty(\D^n)$ and  $$\lim_{\varphi(z)\to \partial \D^n}|\psi(z)|\sum_{j=1}^n\log\frac{1+\mod{\varphi_j(z)}}{1-\mod{\varphi_j(z)}}=0.$$
\end{enumerate}
\end{theorem}

\begin{proof} The implication $(a)\Longrightarrow (b)$ is obvious.  We now show $(b)\Longrightarrow (c)$.  Since $W_{\psi,\varphi}:\Bloch_{0*}(\D^n)\to H^{\infty}(\D^n)$ is bounded, by Theorem~\ref{Polydiskcharaboundedness}, $\psi\in H^\infty(\D^n)$. Suppose there exists a sequence $\{z^{(k)}\}$ in $\D^n$ such that $\varphi(z^{(k)})\to\partial D^n$ as $k\to \infty$. Then, there is a number $j\in\{1,\dots,n\}$ such that $|\varphi_j(z^{(k)})|\to 1$ as $k\to\infty$. Since (\ref{polybddcond}) holds, it follows that
\begin{eqnarray}\label{zerolimit} \lim_{k\to\infty}|\psi(z^{(k)})|=0.\end{eqnarray}  For any such index $j$ and for $z\in\D^n$, define $$f_k(z)=\frac{\left(\hbox{Log}\frac4{1-z_j\overline{\varphi_j(z^{(k)})}}\right)^2}{\log\frac4{1-|\varphi_j(z^{(k)})|^2}}.$$
As shown in the proof of Theorem~\ref{theorem:wco bounded/compact B_n} for the case of the ball, the sequence $\{f_k\}$ is bounded in $\Bloch(\D^n)$, converges to 0 locally uniformly in $\D^n$ and each function in the sequence is in $\Bloch_{0*}(\D^n)$, since it is holomorphic on the closure of $\D^n$. By the compactness of $W_{\psi,\varphi}:\Bloch_{0*}(\D^n)\to H^\infty(\D^n)$, we obtain
\begin{eqnarray}\nonumber |\psi(z^{(k)})|\log
\frac{1+\mod{\varphi_j(z^{(k)})}}{1-\mod{\varphi_j(z^{(k)})}}&\leq& |\psi(z^{(k)})f_k(\varphi(z^{(k)})|\leq \norm{\psi(f_k\circ \varphi)}_\infty\to 0\end{eqnarray}
as $k\to\infty$.

	Next, assume $j\in\{1,\dots,n\}$ is such that $|\varphi_j(z^{(k)})|\not\to 1$ as $k\to\infty$, so that there exists $r\in (0,1)$ such that $|\varphi_j(z^{(k)})|\leq r$ for all $k\in\N$. Then, by (\ref{zerolimit}) we obtain
$$|\psi(z^{(k)})|\log\frac{1+\mod{\varphi_j(z^{(k)})}}{1-\mod{\varphi_j(z^{(k)})}}\leq \log\frac{1+r}{1-r}|\psi(z^{(k)})|\to 0$$
as $k\to\infty$. Hence, combining the cases when $|\varphi_j(z^{(k)})|\to 1$ or $|\varphi_j(z^{(k)})|\not\to 1$ as $k\to\infty$, we deduce
$$\lim_{k\to\infty}|\psi(z^{(k)})|\sum_{j=1}^n\log\frac{1+\mod{\varphi_j(z^{(k)})}}{1-\mod{\varphi_j(z^{(k)})}}=0,$$
as desired.  Lastly, $(c)\Longrightarrow (a)$ follows at once from (\ref{omegarhobounds}) and Theorem~\ref{suffcondinfty}.
\end{proof}

\section{Further Developments}\label{section:open questions}
In this section, we outline other topics of interest for weighted composition operators not considered in the previous sections.  This list is certainly not exhaustive.  Our intent is to point out some work that has been done in other settings and how it would pertain to the setting of the Bloch space on a bounded homogeneous domain.

\subsection{Isometries} A characterization of the isometric weighted composition operators on the Bloch space of the unit disk is not currently known, although the isometric multiplication operators and the isometric composition operators have been described in Theorem 3.1 of \cite{AllenColonna:08}, and Corollary 2 of \cite{Colonna:05} (see also \cite{MartinVukotic:07}, Theorem~1.1).  These results provide a means by which to construct isometric weighted composition operators.

In higher dimensions, the isometric multiplication operators acting on the Bloch space of a large class of bounded symmetric domains are precisely the constant functions of modulus 1 \cite{AllenColonna:08-I}. Yet, it is not known whether nontrivial isometric multiplication operators exist on a general bounded homogeneous domain. Conditions for a composition operator on the Bloch space of a bounded homogeneous domain to be an isometry were given in \cite{AllenColonna:07}. These conditions allow us to generate nontrivial examples of isometric weighted composition operators on the Bloch space for a large class of domains that have the unit disk as a factor.

\subsection{Spectrum} The spectrum of the multiplication operator on the Bloch space of the unit disk is known (\cite{AllenColonna:08}, Theorem 4.1), while the determination of the spectrum of the composition operator on the Bloch space of the unit disk is still an open problem. The authors determined the spectrum of the isometric composition operators on the Bloch space of the unit disk, and in turn, the spectrum of a large class of isometric weighted composition operators on the Bloch space of the unit disk. The spectrum of a non-isometric weighted composition operator has not been determined for a general class of symbols.

In higher dimensions, the spectrum of a class of isometric composition operators on the Bloch space of the unit polydisk has been determined (\cite{AllenColonna:07}, Theorem 7.1).  In \cite{AllenColonna:08-I} (Theorem~5.1), we showed that the spectrum of a multiplication operator on the Bloch space of a bounded homogeneous domain is the closure of the range of its symbol. On the other hand, in \cite{AllenColonna:08-I} (Corollary~3.6), we proved that the only bounded multiplication operators on the Bloch space of the polydisk $\D^n$ (for $n\geq 2$) are those whose symbol is constant. Thus, the only isometric multiplication operators are those induced by constant functions of modulus one and the corresponding spectrum reduces to the value of that constant.

\subsection{Essential norm}
The \textit{essential norm} of a bounded operator $T$ is the distance from $T$ to the compact operators, i.e. $\norm{T}_e = \inf \{\norm{T-K} : K \text{ is compact}\}$.  In \cite{MacCluerZhao:03}, MacCluer and Zhao established estimates on the essential norm of a weighted composition operator acting on the Bloch space of the unit disk.  They showed that $$\max\left\{A_{\psi,\varphi},\frac{1}{6}B_{\psi,\varphi}\right\} \leq \norm{W_{\psi,\varphi}}_e \leq A_{\psi,\varphi} + B_{\psi,\varphi},$$ where
$$\begin{aligned}
A_{\psi,\varphi} &= \lim_{s \to 1}\sup_{\mod{\varphi(z)}>s}\mod{\psi(z)}\mod{\varphi'(z)}\frac{1-\mod{z}}{1-\mod{\varphi(z)}}, \text{ and}\\
B_{\psi,\varphi} &= \lim_{s \to 1}\sup_{\mod{\varphi(z)}>s}\mod{\psi'(z)}(1-\mod{z}^2)\log\frac{1}{1-\mod{\varphi(z)}^2}.
\end{aligned}$$

	An estimate on the essential norm of the weighted composition operators on the Bloch space of the polydisk has been given by Li in \cite{Li:08}. To date, no results have appeared on the essential norm of a weighted composition operator acting on the Bloch space of the unit ball or other types of bounded homogeneous domains.

\bibliographystyle{amsplain}
\bibliography{references.bib}
\end{document}